\batchmode
\makeatletter
\def\input@path{{/Users/andrewpowell/Downloads/}}
\makeatother
\documentclass[oneside,english]{amsart}
\usepackage[T1]{fontenc}
\usepackage[latin9]{inputenc}
\usepackage{amsthm}
\usepackage{amssymb}

\makeatletter
\numberwithin{equation}{section}
\numberwithin{figure}{section}
\theoremstyle{plain}
\newtheorem{thm}{\protect\theoremname}

\makeatother

\usepackage{babel}
\providecommand{\theoremname}{Theorem}

\begin{document}
\title{Construction and Set Theory}
\author{Andrew Powell}
\address{Dr. Andrew Powell, Honorary Senior Research Fellow, Institute for
Security Science and Technology, Level 2 Admin Office Central Library,
Imperial College London, South Kensington Campus, London SW7 2AZ,
United Kingdom.}
\email{andrew.powell@imperial.ac.uk}
\begin{abstract}
This paper argues that mathematical objects are constructions and
that constructions introduce a flexibility in the ways that mathematical
objects are represented (as sets of binary sequences for example)
and presented (in a particular order for example). The construction
approach is then applied to searching for a mathematical object in
a set, and a logarithm-time search algorithm outlined which applies
to a set $X$ of all binary sequences of length ordinal $\beta$ with
a binary label appended to each sequence to indicate that sequence
is a member of $X$ or not. It follows that deciding membership of
a set for a given binary sequence of length of binary sequence of
cardinal length $\beta$ takes $\beta+1$ bits, which is shown to
be equivalent to the Generalised Continuum Hypothesis on the assumption
that information is minimized when a mathematical object is created.
\end{abstract}

\maketitle

\section{Philosophical Introduction}

This is a short paper about set theory as a foundation for mathematics.
It is not my intention to repeat what many authors have already written
on the subject of set theory, so there is no discussion of the iterative
conception of sets, forcing or limitation of size arguments, and only
a mention of large cardinal axioms as a complexity measure.\footnote{See \cite{Jech2002} for an encyclopedic overview of set theory up
to the millennium and \cite{Boolos1971}, \cite{Shoen1977} for very
readable introductions to the iterative conception of set, which remains
the standard motivation for set theory in terms of motivating the
axioms of first-order Zermelo Fraenkel set theory. \cite{Hallett86}
gives an excellent background in the development of the concept of
set, while \cite{Hellman89} gives a structuralist interpretation
of set theory that is still unsurpassed in clarity. Large cardinal
axioms (axioms asserting the existence of infinite cardinal numbers
with certain defining properties that are not theorems of first-order
Zermelo Fraenkel set theory) have a vast literature, but \cite{Kanamori94}
is a good introduction.} Rather the aim of this paper is to convince the reader about a certain
way of looking at mathematics, which has some implications for set
theory. That way of looking at mathematics owes something to information
theory and computer science, and a great deal to P. Lorenzen's notion
of construction (see \cite{Lorenzen87} and \cite{Powell97}). \\
\\
The basic idea is that all of the objects and activities of mathematics
are constructed by functions, and that the existence of the functions
enables objects (including sets) to be defined. To give a simple example,
the function of successor defines the set of natural numbers (subject
to the condition that there is an initial number, 0, and the successor
function does not output 0) given that the construction defines the
smallest such set because an agent with unbounded but finite resource
would construct exactly the set of the natural numbers.\footnote{Strictly, in terms of an ontology each mathematical ``object'' is
really a function (or type) over a set of concrete individuals, because
there is an issue of non-unique types, such as in the statement ``1,
2 and 3 are 3 numbers''.} Moreover constructions can also be carried out on much larger sets
than the set of natural numbers, in much the same way as intuitionists
admit for natural numbers and real numbers, namely by free choice.\footnote{See for example \cite{Troelstra83}.}
The axiom of choice in the form of the well-order-ability of any well-founded
set is a key principle of infinite construction, and is constructive
because an agent with sufficient (\emph{i.e. }infinite) resource could
choose elements successively and at infinite limits form the sequence
of all elements chosen so far. If one accepts infinite constructions,
then the truth or falsehood of any proposition of first-order set
theory follows. For example, the truth of a quantified proposition
has a clear inductive construction in terms of a sequence of truth
values of its subformulas that follows the constant true sequence
or constant false sequence of truth values or that does not follow
those sequences.\footnote{For example, $(\forall x)P(x)$ is true in a model $M$ if $\{a:a\in M\}$
can be well ordered as $\{a_{\alpha}:\alpha<\aleph\}$ using the axiom
of choice and the truth values of $<P(a_{\alpha}):\alpha<\aleph>$
form a constant sequence of value ``true'' of length $\aleph.$
The constant sequence of value ``false'' corresponds to $(\forall x)\neg P(x)$
and not following constant sequence of value ``false'' corresponds
to $(\exists x)P(x).$} While constructions determine how objects come to exist, that does
not mean that relationships between the objects cannot exist that
were not intended as part of the construction. Mathematics does not
need to be predicative (\emph{i.e.} defining sets in stages only in
terms of sets that are already defined) provided the rule or process
of construction is clear (which in my view includes the process of
choosing members of a set).\footnote{This is a deviation from the view of Lorenzen and the school that
includes H. Poincaré, H. Weyl and S. Feferman, see \cite{Feferman87}
for example. } As truth is well defined, the logic of mathematics does not need
to be constructive or intuitionistic. However, according to this view
the objects of mathematics are no more than constructions, and we
should not imagine that they exist independently of the process of
their construction. The objects of mathematics are possibilities of
construction, in the modal-structural sense of \cite{Hellman89},
and it is the clarity of their rules of construction that grants them
existence. \\
\\
All constructions create information. It is reasonable to suppose
that Ockham's Razor applies: when an object is created, the amount
of information created with it is the least possible to be consistent
with other objects.\\
\\
One problem with this approach is the status of these agents with
infinite resources (actually bounded by some infinite ordinal). I
do not claim that such agents exist in our physical world, but I do
claim that their existence is possible if a rule of construction that
an agent uses is clear. In the same way that Euclid's proof of the
infinite of primes gives a bound on finding the next prime in the
sequence of prime natural numbers, and thereby shows that the number
of prime natural numbers is infinite even though there are only finitely
many atoms in the universe, rules of construction that require infinite
resources can have interesting properties that help frame our theories
of the physical world.\\
\\
This may be all very well as a philosophical position (or not of course),
but what practical value does it have? Put briefly, the value of this
position is the recognition that mathematicians have freedom to represent
a set of objects as they wish subject to the constraints of the construction,
including the presentation of the set in terms of ordering. That is
to say, if a mathematical object does not come equipped with its own
intrinsic ordering, an ordering can be added without affecting the
intrinsic properties of the mathematical object. It turns out that
freedom to present and represent mathematics does have practical consequences.

\section{Search for a Member of a Set}

As an example of the constructive nature of mathematics, consider
the question of what it means to search for a member of a set. In
theory, if we represent the members of a set as binary sequences (or
\emph{bitstrings} for short), then you could read the bitstring and
then append a label (say 1) to the bitstring if the bitstring were
a member of the set and another label (say 0) if the bitstring were
not a member of the set.\footnote{This is possible by fixing an enumeration of a set $X$,  $\langle x_{\alpha}:\alpha<\aleph\rangle$
(by the Axiom of Choice), and for any subset $Y\subseteq X$ forming
the binary \emph{$\aleph$-}sequence $\langle b_{\alpha}:(x_{\alpha}\in Y\rightarrow b_{\alpha}=1)\vee(x_{\alpha}\notin Y\rightarrow b_{\alpha}=0)\rangle$,
where the ordinal index of any member $y\in Y$ is taken from the
enumeration of $X$ (which includes all members of $Y$). Thus a subset
of $X$ can be identified with a binary $\aleph$-sequence, and a
set of subsets of $X$ can be identified with a set of binary $\aleph$-sequences. } In general we would have to rely on an oracle to decide whether a
set defined in this way were (equivalent to) the same set as a defined
by a property of the members, but this lack of decidability is a problem
with properties rather than with sets. We can say that if a set comprises
bitstrings that each have length of least upper bound an ordinal $\alpha$
of cardinal number $\aleph$, then the amount of information in searching
for a member of the set is, adding 1 to the length of the sequence
for the binary label, $\alpha+1$. In practice, for any reasonably
large set we will be faced with a lot of bitstrings, and have no way
to search for a particular bitstring $x$ other than to enumerate
the set of bitstrings somehow. Let us suppose (using our freedom of
construction) that we can linearly order lexicographically (written
$\preceq$)\footnote{$z\preceq y$ if $(\exists\alpha<\aleph)[(z_{\alpha}<y_{\alpha})\wedge(\forall\beta<\alpha)(z_{\beta}=y_{\beta})$
or $(\forall\beta<\aleph)(z_{\beta}=y_{\beta})$} the members of the set such that there is a least upper bound and
greatest lower bound (in terms of bitstrings of length $\aleph$)
for the set as a whole and we can assign a distance between any two
members of the set. It is reasonable to suppose that a set can be
presented already linearly ordered, not when we are faced with a list
to sort, but when we can choose how to present a set in the first
place.\\
\\
To justify our assumptions, we can define an \emph{interval} $X$
of binary sequences of length ordinal $\beta$ as a set of all such
binary $\beta$-sequences (binary sequences of length $\beta$) with
the properties that every path through the tree of sequences from
root to leaves is a branch of the tree, \emph{i.e.} $(\forall f:\beta\rightarrow\{0,1\})((\forall x)(x\in f\rightarrow x\in\in X)\rightarrow(f\in X))$,
where $x\in\in y$ is defined as $(\exists z)(x\in z\wedge z\in y)$.
Intervals defined in this way are not uniquely determined by ordinal
$\aleph$ as the tree could have gaps between the sequences, but it
is possible to make them unique by stipulating that for interval $X$,
$(\forall f:\beta\rightarrow\{0,1\})(f\in X)$. We can also stipulate
that the root represents $0.$, so that in a sense the interval represents
the maximal interval from 0 to 1 comprising binary $\beta$-sequences.
Intervals of this type are written $([0,1])(\beta)$. To justify that
any two members $x,y$ of $([0,1])(\beta)$ can be assigned a distance
$d(x,y)$ to be constant 0 $\beta$-sequence with 1 at the position
where $x$ and $y$ first differ (read from 0. onwards). Then $d$
can be seen to be a generalised\footnote{$d$ is a generalised ultrametric because distances are not real numbers
but binary $\beta$-sequences. } ultrametric (\emph{i.e.} $max(d(x,y),d(y,z))\ge d(x,z)$).\footnote{To see this, fix labels $x,y,z$ arbitrarily. Then if $x$ splits
from $y$ before $x$ splits from $z$, then $d(x,y)\ge d(x,z)$ and
$d(y,z)=d(x,y)$, so $max(d(x,y),d(y,z))=d(x,y)\ge d(x,z)$. If $x$
splits from $y$ after $x$ splits from $z$, then $d(x,z)\ge d(x,y)$
and $d(y,z)=d(x,z)$, so $max(d(x,y),d(y,z))=d(x,z)\ge d(x,z)$. Finally,
if $x$ splits from $y$ at the same position that $x$ splits from
$z$, then $d(x,z)=d(x,y)$ and $d(y,z)\le d(x,y)$, so $max(d(x,y),d(y,z))=d(x,y)\ge d(x,z)$.
These inequalities are not strict and allow for the cases of $x=y$,
$y=z$ or $z=x$. } \\
\\
It is possible to losslessly compress any binary $\beta$-sequence
to a binary $\beth$-sequence where $\beth$ is a cardinal $\beth\le\beta<\beth+1$.
We can thus represent any set of $\beta$-sequences $X$ as $\subseteq([0,1])(\beth)$.
But we actually want the construction below to use sets such that
each $\beth$-sequence is labelled with a 1 (if $x\in X$) and 0 (if
$x\notin X$). These \emph{labelled sets} of binary $\beth$-sequences
are then $\subset([0,1])(Ord(\beth)+1)$ such that the set of binary
$\beth$-sequences without the labels $=([0,1])(\beth)$. We will
write the labelled set of binary $\beth$-sequences corresponding
to set $X$ as $L(X)$. \\
\\
Any set $X$ of size $\le2^{\aleph}$ can be searched for the bitstring
$x$ of length $\aleph$ in $\aleph$ steps by representing the set
$X$ by the labelled set $L(X)$ and then dividing $([0,1])(\aleph)$
into two equal intervals (which is possible whether the midpoint is
$\in X$ or not), choosing the interval that contains $x$ based on
the value of the next bit of $x$ (because $a\preceq x\preceq b$
for $a,b$ the lower and upper limits of the interval) and iterating
$\aleph$ times (taking the intersection of intervals at any limit
ordinal stages), and checking the label of $x$ in $L(X)$ at the
$Ord(\aleph)+1$-th step. \\
\\
A shorthand way to express the search for the bitstring $x$ is to
note that there are $\le2^{\aleph}$ bitstrings to be searched, but
that binary search runs in logarithmic time. Therefore there are $\simeq log_{2}(2^{\aleph})=\aleph$
bits of information in the search for $x\in X$. In the simplest case
of the real numbers, we can see that the search method amounts to
binary search for a binary $\omega$-sequence in a labelled set that
extends the closed interval $[0,1]$. That us to say, every binary
$\omega$-sequence is represented (starting with $0.$ in the case
of $[0,1]$) and every $\omega$-sequence has an extension at position
$\omega+1$ which states whether $x\in X$, where $X$ is coded as
a set of binary $\omega$-sequences. It is clear that $x\in X$, for
$X$ a set of real numbers, can be decided in $\le\omega+1$ steps.
But does that mean that the set of real numbers is the closed interval
$[0,1]$? No, but it does mean that the set of real numbers are represented
by $[0,1]$ insofar as purely set theoretic properties, such as cardinality,
are concerned.\\
\\
This enumeration (well-ordering) of intervals can also be regarded
as an enumeration of members of the intervals. Members of the intervals
may be members of $X$ but they do not have to be. For definiteness
and balance we alternately choose $\aleph$-sequences in $X$ and
$([0,1])(\aleph)-X$ as successive elements of the enumeration as
far as possible (ending when an interval has all members $\in X$
or $\notin X)$, and we see that there are $\le(Ord(\aleph)+1)\times Ord(\aleph)+1$
steps to decide $x\in X$. Thus for any given binary $\aleph$-sequence
$x$ there is an enumeration of $X$ and $([0,1])(\aleph)-X$ that
takes $<\aleph+1$ steps to decide $x\in X$. We call this last statement
({*}).\\
\\
The statement ({*}) is not the strongest statement we can make about
searching for members of $X$. It is also true that (+) $(\exists f:\aleph+1\times([0,1])(\aleph)\rightarrow\{0,1\})(\forall x)(\exists\alpha)[(f(\alpha,x)=1\rightarrow x\in X)\wedge(f(\alpha,x)=0\rightarrow x\notin X)]$,
since $f(\alpha,x)=x_{\alpha},$ the last member of the labelled sequence
$x=\langle x_{\beta<\alpha}\rangle\parallel\langle x_{\alpha}\rangle$
for $\aleph<\alpha<\aleph+1$ trivially satisfies (+). (+) is actually
equivalent to ({*}) by an application of the axiom of choice. \\
\\
The amount of information of bits in function $f(\alpha)=(\lambda x)f(\alpha,x)$,
for functional abstraction operator $\lambda$, is at least $\aleph+1$
because any $\aleph$-bit binary code for $f(\alpha)$ would also
be a code for some $x\in([0,1])(\aleph)$. We can express this by
means of a diagonal function $d(\left\lceil y\right\rceil ):=1-\left\lceil y\right\rceil (\left\lceil y\right\rceil )$
for $\left\lceil y\right\rceil $ a $\aleph$-bit code for a function
of $\aleph$-bits, and note that we get a contradiction if we put
$d:=\left\lceil y\right\rceil $ unless the number of bits in $d$
is greater than the number of bits in $\left\lceil y\right\rceil $.
This implies that the number of bits in $f$ (where $f:=(\lambda\alpha)f(\alpha)$)
is at least $\aleph+1$.\\
\\
We can say (++) that for infinite cardinal $\aleph$ $(\lambda x\in2^{\aleph})(x\in X)$,
the concept of being a member of set $X$, contains $\aleph+1$ bits
of information, and any $x\in2^{\aleph}$ can be decided in $<\aleph+1$
bits. The reason this is true is effectively the diagonal argument
again, because otherwise the $\aleph$-bit binary code for ($\lambda x\in2^{\aleph})(x\in X)$
would also be a code for some $x\in X$. (++) is consistent with application
and abstraction operations in the lambda calculus, since application
and abstraction apply in this case to generic $\aleph$-sequences.
In fact if we were to choose to represent a generic $x\in2^{\aleph}$
by an $\alpha$-sequence, where $\aleph\le\alpha<\aleph+1$, we see
that $\aleph+1$ is a natural information measure for ($\lambda x\in2^{\aleph})(x\in X)$
as it is the least upper bound of $\alpha$. \\
\\
Principle (++) is equivalent to the Generalised Continuum Hypothesis
(GCH) for $\aleph\ge\aleph_{0}$, as it is a choice principle that
limits the number of bits in deciding whether any $x\in2^{\aleph}$
by a function to $<\aleph+1$ bits in any decision process.\\
 
\begin{thm}
\label{thm:GCH-is-equivalent} GCH is equivalent to (++) for $\aleph\ge\aleph_{0}$. 
\end{thm}

\begin{proof}
Assume GCH. and fix a binary $\aleph$-sequence $x$. Then if $x\in X$
then by GCH $x$ will be decided in $<\left|X\right|\le2^{\aleph}=\aleph+1$
bits. While if $x\notin X$ then $x$ will be decided in $<\left|([0,1])(\aleph)-X\right|=2^{\aleph}=\aleph+1$
bits. In either case then $x\in X$ can be decided $<\aleph+1$ steps,
\emph{i.e.} decided in $\le\aleph$ steps since $\aleph$ is a cardinal.
But if $x\in X$ can be decided in $\le\aleph$ steps, then it can
be decided in $\le Ord(\aleph)+1$ steps. \\
\\
Now assume (++) and that GCH is false, \emph{i.e.} \emph{$X$ }has
cardinality $\aleph<c<2^{\aleph}$, and fix a binary $\aleph$-sequence
$x$. Then if $x\in X$, we would always find $x$ in $<c$ bits by
enumeration since there are $\left|([0,1])(\aleph)-X\right|=2^{\aleph}$
members of $([0,1])(\aleph)-X$ to be enumerated otherwise (and $c<2^{\aleph})$.
We can now check that $c=\aleph+1$ is consistent with (++), but $c>\aleph+1$
leads to $x$ being decided almost always in $\ge\aleph+1$ bits (contradiction)
and $\aleph+1>c$ leads to $\aleph+1>c>\aleph$ (contradiction). If
$x\notin X$, then we could either enumerate all $c$ members of $X$
or $<2^{\aleph}$ members of $([0,1])(\aleph)-X$. But enumerating
all of $c$ members of $X$ contradicts (++) because $c=\aleph+1$
leads to $x$ being decided in $\aleph+1$ steps (contradiction),
$c>\aleph+1$ leads to $x$ being enumerated almost always in $>\aleph+1$
steps (contradiction) and $\aleph+1>c$ leads to $\aleph+1>c>\aleph$
(contradiction). The remaining possibility if $x\notin X$ is that
$x$ is enumerated in $<2^{\aleph}$ bits in $([0,1])(\aleph)-X$.
Then $\aleph+1=2^{\aleph}$ is consistent with (++), and $\aleph+1<2^{\aleph}$
leads to $x$ being decided almost always in $\ge\aleph+1$ steps
(contradiction) and $\aleph+1>2^{\aleph}$ contradicts Cantor's theorem
that $\aleph+1\le2^{\aleph}$ (contradiction). Since $X$ is not empty
and $\ne([0,1])(\aleph)$ because $X$ has cardinality $c$, then
both $c=\aleph+1$ and $\aleph+1=2^{\aleph}$ are witnessed as $x$
and the associated enumerations vary; hence $c=2^{\aleph}$ (contradiction).
Hence GCH is true.
\end{proof}
The reason why a statement like GCH that is independent of first-order
Zermelo-Fraenkel set theory turns out to be true for almost all sets\footnote{The proof of Theorem 1 uses ``almost always'' in its arguments.
There will be a very small proportion ($\aleph/2^{\aleph})$ of sets
where the equivalence does not hold. } if (++) is true is that (++) requires  a very rich theory to be true.
If we were to measure the complexity of a decision problem by the
size of any set (\emph{i.e.} possibly ``a large cardinal'')\footnote{If the axiom of choice is assumed, then the size a set is its only
distinguishing feature, since all sets are well-orderable and are
isomorphic to ordinals; and ordinals can be losslessly compressed
to cardinals.} that is needed to solve the decision problem by deduction from the
axioms of a first-order theory of sets,\footnote{See \cite{Davis2006a} 417 for an example from the work of H. Friedman
of statements that can be encoded in first-order arithmetic that require
a large cardinal axiom. } then (++) indicates that for infinite cardinal $\aleph$, $\aleph$
rather than a large cardinal would be the measure of complexity. In
Zermelo-Fraenkel set theory with the axiom of choice (ZFC), a proof
is the construction of set (which in first-order ZFC corresponds to
a formula in the language of ZFC). A set $x$ can be identified with
an enumeration of $x$ by a (one-to-one) function $f$ such that $f(\alpha)=x$
for some ordinal $\alpha$, which follows by the well-ordering theorem
(requiring the axiom of choice). Thus the cardinality of the set (as
the least ordinal) needed to decide a decision problem is a natural
and useful measure of the complexity of the decision problem. We can
conclude that (++) is not compatible with decidability by a first-order
deductive theory (that uses set cardinality as a complexity measure
of decidability), but is compatible with truth in an initial segment
of the Von Neumann hierarchy of pure sets, $V$. $V$ itself is a
class model of first-order set theory. \\
\\
Labelled sets are a good way to see the power of the decision criterion
(++). Labelled sets clearly represent a standard binary coding of
any set, but with the advantage that it is easy to tell which binary
sequences are members of the set or not. There are uncountably more
labelled sets than there can be sets defined in any countable formal
language of set theory, because each $\subseteq([0,1])(\aleph)$ has
labels for all its members and non-members (which is not true for
membership defined by means of formulas through the axiom schemas
of separation or replacement). In fact we can see that all sets $\subseteq([0,1])(\aleph)$
can be labelled for all cardinals $\aleph$. It is worth noting that
labelled sets do not satisfy the axioms of first-order Zermelo-Fraenkel
set theory, because functions cannot be applied to labels in the same
way as to the data that they label; but the axioms could be easily
modified by stripping out the labels (\emph{i.e.} the $Ord(\aleph)+1$-th
nodes), applying the function to binary sequences of length $\aleph$
and adding back the labels. That is, if $L(X)$ is a labelled set
of binary $\aleph$-sequences then we can form the labelled set $\{\langle y,0\rangle:y\in([0,1])(\aleph)\}\sqcup\{\langle y,1\rangle:\langle x,1\rangle\in L(X)\rightarrow y=f(x)\}$,
where $\langle\rangle$ is a $Ord(\aleph)+1$-sequence and $\sqcup$
is a union operator with the property that $\langle y,0\rangle\sqcup\langle y,1\rangle=\langle y,1\rangle$.
It is clear though that labelled sets preserve what sets can be formed
in initial segments of $V$.

\section{Conclusions}

I think the example of search for a member of a set shows, at least
in principle, that taking mathematical objects as constructions (for
example, labelled sets) which can be represented and ordered in different
ways has mathematical consequences. The alternative to the labelled
set approach discussed above is to suppose that there are sets which
in principle we cannot define (by means of finite formulas) and of
which we are not even permitted to see their shadows. 

\bibliographystyle{plain}
\bibliography{0_Users_andrewpowell_Downloads_hyperc1}

\end{document}